\documentclass[10pt]{amsart}
\usepackage{amsfonts,amssymb,amscd,amsmath,enumerate,verbatim,calc}
\usepackage{mathrsfs}
\usepackage{tikz-cd}
\numberwithin{equation}{section}
\newtheorem{theorem}{Theorem}[section]
\newtheorem{corollary}[theorem]{Corollary}
\newtheorem{lemma}[theorem]{Lemma}

\newtheorem{proposition}[theorem]{Proposition}
\newtheorem{definition}[theorem]{Definition}

\newtheorem{example}[theorem]{Example}

\begin{document}

\title[Representations and deformations of 3-Hom-$\rho$-Lie algebras]
 {Representations and deformations of 3-Hom-$\rho$-Lie algebras}

\bibliographystyle{amsplain}

\author[E. Peyghan, Z. Bagheri, T. CAKMAK and A. Gezer]{E. Peyghan $^{1}$, Z. Bagheri  $^{2}$,  I. Gultekin  $^{3}$ and A. Gezer $^{4}$}
\address{$^{1,2}$ Department of Mathematics, Faculty of Science, Arak University,
	Arak, 38156-8-8349, Iran.}
\email{e-peyghan@araku.ac.ir  $^{1}$, z-bagheri@phd.araku.ac.ir  $^{2}$}
\address{$^{3,4}$ Ataturk University, Faculty of Science, Department of Mathematics,
	25240, Erzurum-Turkey.}
\email{igultekin@atauni.edu.tr   $^{3}$, agezer@atauni.edu.tr  $^{4}$}


\keywords{3-Hom-$\rho$-Lie algbras, Abelian extensions, Deformations, Representations.}

\subjclass[2010]{11R20, 17B10, 16W25, 17B56, 17B70, 17B75.}


\begin{abstract}
The aim of this paper is to introduce 3-Hom-$\rho$-Lie algebra structures generalizing the
algebras of 3-Hom-Lie algebra. Also, we investigate the representations and deformations theory of this type pf Hom-Lie algebras. Moreover, we introduce the definition of extensions and abelian
extensions of 3-hom-$\rho$-Lie algebras and show that associated to any abelian extension, there is a representation and a 2-cocycle.
\end{abstract}

\maketitle







\section{Introduction}
The structure of Hom-Lie algebra appeared first as a generalization of Lie algebra by Hartwig, Larsson and Silvestrov in \cite{J D S}. In 1994, the concept of  $\rho$-Lie algebra or Lie color algebra introduced by Bongaarts \cite{BP1} and then in 1998, Scheunert and Zhang introduced the cohomology theory of Lie
color algebras in \cite{MR}. Also, in 2012, Yuan \cite{LY} introduced the notion of a hom-Lie color algebra which can be viewed as an extension of Hom-Lie superalgebras to $G$-graded algebras, where $G$ is any abelian group. In 2015, Abdaoui, Ammarto and Makhlouf defined representations
and a cohomology of the Hom-Lie color algebra in \cite{AAM}. After two years, in 2017, $T^*$-extensions and abelian extensions of the Hom-Lie color algebras are studied by Bing Sun, Liangyun Chen and Yan Liu in \cite{BLY}.

Filippov, in 1985 introduced a concept that is called $n$-Lie algebra. These Lie algebras are represented with various names such as Filippov algebra, Nambu-Lie algebra, Lie $n$-algebra. The notion of $n$-Lie algebra has close relationships with many fields in mathematics
and mathematical physics, for their applications refer to \cite{TL, AG, HH, JN, JN1}. The cohomology theory and deformation theory for $n$-Lie algebras was introduced respectively by Takhtajan and Gautheron in \cite{TL1, D1, PG}. H. Ataguema, A. Makhlouf and S. Silvestrov in \cite{HAS} have introduced the notion of 3-hom-Lie
algebras and representations and module-extensions of 3-hom-Lie algebras have investigated by Y. Liu, L. Chena and Y. Ma in \cite{LCM}. The notion of 3-Lie colour algebras have introduced by T. Zhang and have studied Cohomology and deformations of 3-Lie colour algebras by him (see \cite{T}, for more details).

In this paper, we introduce the notion of 3-Hom-Lie colour algebras or 3-Hom-$\rho$-Lie algebras and study the representation and deformation theory of this kind of Hom-Lie algebras.

This paper is arranged as follows: In Section 2, we recall some necessary background knowledge including $\rho$-commutative and
Hom-$\rho$-Lie algebras. In the next, we discuss about the 3-Hom-$\rho$-Lie algebras and define representations, modules, $\phi^k$derivations of it and show that representations and modules of 3-Hom-$\rho$-Lie algebras are equivalent. This section also contains the $T^{\star}$-extension of 3-Hom-$\rho$-Lie algebras. Section 3 is contained abelian extensions of 3-Hom-$\rho$-Lie algebras and the reader will get some results in this case. In this section we show that associated to any abelian extension, there is a representation and a 2-cocycle.  Section 4 is devoted to discuss about deformations and the Hom Nijenhuis operator of 3-Hom-$\rho$-Lie algebras. Furthermore, we show that $\omega$ generates a $t$-parameter infinitesimal deformation of the 3-Hom-$\rho$-Lie algebra $A$ is equivalent to 3-Hom-$\rho$-Lie algebra which is 1-cocycle of $A$ with coefficients in the adjoint representation.


\section{3-hom-$\rho$-lie algebras}
In this section, we summarize some definitions concerning $\rho$-commutative algebras and Hom-$\rho$-Lie algebras. We also introduce the notion of 3-Hom-$\rho$-Lie algebras. Representations, modules, $\phi^k$-derivations and some results about them are studied in this section.

Let $A$ be an associative and unital algebra over a field $k$ ($k = \mathbb{R}$ or $k = \mathbb{C}$), grading by an abelian group $(G, +)$ that is  the vector space $A$ has a $G$-grading $A = \oplus_{a\in G}A_a$
such that $A_aA_b\subset A_{a+b}$. A map
$\rho:G\times G\longrightarrow k^{\star}$ is called a two-cycle if the following conditions hold
\begin{align}
&\rho(a,b) =\rho(b,a)^{-1},\quad a,b\in G,\\
&\rho(a+b,c) =\rho(a,c)\rho(b,c),\quad a,b,c\in G.
\end{align}
The above conditions say that 
$\rho(a,b)\neq 0$, $\rho(0,b)=1$ and $\rho(c,c)=\pm 1$ for all $a,b,c\in A$ , $c\neq 0$.

Let us denote by $Hg(A)$ the set of homogeneous elements in $A$. The $\rho$-commutator of two homogeneous elements $f,g$ is
\begin{equation}\label{111}
[f, g]_{\rho} = fg-\rho(|f|, |g|)gf,
\end{equation}
where $|f|$ denotes the $G$-degree of a (non-zero)  homogeneous element $f\in A$.\\
A $\rho$-commutative algebra is a $G$-graded algebra $A$ with a given two-cycle
$\rho$ such that $f g =\rho(| f |, |g|)g f $ for all homogeneous elements $ f$ and
$g$ in $A$ (i.e., $[f,g]_{\rho}=0$).
\begin{definition}
A 2-Hom-$\rho$-Lie algebra or for simply a Hom-$\rho$-Lie algebra is a $G$-graded vector space $A$ together with a bilinear map $[.,.]_{\rho}:A\times A\longrightarrow A$, a two-cycle $\rho$ and a linear map $\phi:A\longrightarrow A$ satisfying the following conditions
\begin{align*}
&\bullet |[f,g]_{\rho}|= | f|+| g|,\\
&\bullet [f, g]_{\rho} = -\rho(f,g)[g, f]_{\rho},\\
&\bullet \rho(h,f)[\phi(f), [g, h]_{\rho}]_{\rho}+\rho(g,h)[\phi(h), [f, g]_{\rho}]_{\rho}+\rho(f,g)[\phi(g), [h, f]_{\rho}]_{\rho} = 0.
\end{align*} 
The third condition is equivalent to 
$$[\phi(f), [g, h]_{\rho}]_{\rho}=[[f, g]_{\rho},\phi(h)]_{\rho}+\rho(f,g)[\phi(g), [f, h]_{\rho}]_{\rho}.$$
\end{definition}
\begin{definition}
A quadruple
$ (A, [.,.,.], \rho,\phi)$ 
consisting of a  $G$-graded vector space
 $A =\bigoplus_{a\in G} A_a$,
 a trilinear map
 $[.,.,.]: A \times A\times A \longrightarrow A$,
a two-cycle
 $\rho : G\times G\longrightarrow k^{\ast}$
 and an even linear map
$\phi : A\longrightarrow A$ is called a 3-Hom-$\rho$-Lie algebra if the following condition are satisfied
\begin{align*} 
&(1)~~|[f_1,f_2,f_3]|=|f_1|+ |f_2| +|f_3|,\\
&(2)~~[\phi(f_1),\phi(f_2),[g_1,g_2,g_3]]=[[f_1,f_2,g_1],\phi(g_2),\phi(g_3)]-\rho(f_1+f_2,g_1)[\phi(g_1),[f_1,f_2,g_2],\phi(g_3)]\\
&\ \ \ \ \ \ \ \ \ \ \ \ \ \ \ \ \ \ \ \ \ \ \ \ \ \ \ \ \ \ \ \ \ \ \ \ +\rho(f_1+f_2, g_1+g_2)[\phi(g_1),\phi(g_2),[f_1,f_2,g_3]].
\end{align*}
The second property is called $\rho$-fundamental identity.
\end{definition}
Note that, the bracket introduced in the above definition has the $\rho$-skew symmetry property with respect to the displacement of every two elements of itself.
\begin{definition}
A 3-Hom-$\rho$-Lie algebra  $(A, [.,.,.],\rho,\phi)$ is said to be multiplicative if $\phi$ is a Lie algebra morphism, i.e. for $f,g,h\in A$, $\phi[f,g,h]=[\phi(f),\phi(g),\phi(h)]$, regular if $\phi$ is an automorphism for $[.,.,.]$, and involutive if $\phi^2 = Id_A$.
\end{definition}
\begin{definition}
Let $(A,[.,.,.]_A,\phi)$ and $(B,[.,.,.]_B,\psi)$ be two 3-Hom-$\rho$-Lie algebra. A linear map $\alpha:A\longrightarrow B$ is said to be a morphism of 3-Hom-$\rho$-Lie algebras if
$$\alpha[f,g,h]_A=[\alpha(f),\alpha(g),\alpha(h)]_B,$$
for all $f,g,h\in A$ and 
$$\alpha\circ\phi=\psi\circ\alpha.$$
\end{definition}
Let us denote by $\vartheta=:\{(f,\alpha(f))| f\in A\}\subseteq A\oplus B$ the group of linear maps $\alpha:A\longrightarrow B$.

Let $(A,[.,.,.],\rho,\phi)$ be a multiplicative 3-Hom-$\rho$-Lie algebra. We define the following operation on the fundamental set $\mathcal{L}=\wedge^2A$ by
\begin{equation}\label{123}
[(f_1,f_2), (g_1,g_2)]_{\mathcal{L}}=([f_1, f_2, g_1],\phi(g_2))+\rho(f_1+f_2,g_1)(\phi(g_1),[f_1,f_2,g_2]).
\end{equation}
If we define the even linear map $\phi_1:\mathcal{L}\longrightarrow \mathcal{L}$ by $\phi_1(f_1,f_2)=(\phi(f_1),\phi(f_2))$, then we have the multiplicative Hom-$\rho$-Lie algebra 
$(\mathcal{L},[.,.]_{\mathcal{L}},\rho,\phi_1)$.

Let $A$ be a 3-Hom-$\rho$-Lie algebra and $V$ be a $G$-graded vector space. Recall that ${\rm End_G(V)} = {\rm Hom_G(V, V )}$ and ${\rm Hom_G(A, V)}$ are $G$-graded vector spaces.
\begin{definition}
Let $A$ be a 3-Hom-$\rho$-Lie algebra, $V$ be a $G$-graded vector space and $\mu$ be a linear map from 
$\mathcal{L}=\wedge^2A$ to ${\rm End_G(V)}$. Then $(V,\mu)$ is called a representation of $A$ with respect to $\beta\in {\rm End_G(V)}$ if the following conditions are satisfied	
\begin{align}
\mu[(f_1,f_2),(g_1,g_2)]_{_\mathcal{L}}\circ \beta&=\mu(\phi_1(f_1,f_2))\mu(g_1,g_2)-\rho(f_1+f_2,g_1+g_2)\mu(\phi_1(g_1,g_2))\mu(f_1,f_2),\label{1}\\
\mu([g_1,g_2,g_3],\phi(f))\circ \beta&=\mu(\phi_1(g_1,g_2))\mu(g_3,f)+\rho(g_1,g_2+g_3)\mu(\phi_1(g_2,g_3))\mu(g_1,f)\label{2}\\
&\ \ \  +\rho(g_1+g_2,g_3)\mu(\phi_1(g_3,g_1))\mu(g_2,f),\nonumber\\
\mu(\phi(g),[f_1,f_2,f_3])\circ \beta&=\rho(g,f_1+f_2)\mu(\phi_1(f_1,f_2))\mu(g,f_3)+\rho(g,f_2+f_3)\rho(f_1,f_2+f_3)\mu(\phi_1(f_2,f_3))\mu(g,f_1)\\
&\ \ \  +\rho(g,f_1+f_3)\rho(f_1+f_2,f_3)\mu(\phi_1(f_3,f_1))\mu(g,f_2),\nonumber\\
\mu(\phi_1(f_1,f_2))\mu(g_1,g_2)&=\rho(f_1+f_2,g_1+g_2)\mu(\phi_1(g_1,g_2))\mu(f_1,f_2)\label{7.1}\\
&\ \ \ +\rho(f_1+f_2,g_1)\mu(\phi(g_1),[f_1,f_2,g_2])\circ\beta +\mu([f_1,f_2,g_1],\phi(g_2))\circ\beta.\nonumber
\end{align}
\end{definition}
\begin{example}
Let $A$ be a 3-Hom-$\rho$-Lie algebra, $V=A$ and $\phi=\beta\in {\rm End_G(A)}$. Then the linear map $ad:A\times A\longrightarrow {\rm End_G(A)}$ defined by $ad(f_1,f_2)(f_3)=[f_1,f_2,f_3]$
is a representation of $A$ with respect to $\beta=\phi$.\\
It is enough to check the conditions \eqref{1} and \eqref{2}. For the condition \eqref{1}, by \eqref{123} and the 
$\rho$-fundamental identity of 3-Hom-$\rho$-Lie algebra $A$, we have 
\begin{align*}
ad[(f_1,f_2),(g_1,g_2)]_{_\mathcal{L}}\phi(g_3) &=ad([f_1,f_2,g_1],\phi(g_2))\phi(g_3)+\rho(f_1+f_2,g_1) ad(\phi(g_1),[f_1,f_2,g_2])\phi(g_3)\\
&=[[f_1,f_2,g_1],\phi(g_2),\phi(g_3)]+\rho(f_1+f_2,g_1)[\phi(g_1),[f_1,f_2,g_2],\phi(g_3)]\\
&=[\phi(f_1),\phi(f_2),[g_1,g_2,g_3]]-\rho(f_1+f_2,g_1+g_2)[\phi(g_1),\phi(g_2),[f_1,f_2,g_3]]\\
&=ad(\phi(f_1),\phi(f_2))ad(g_1,g_2)g_3-\rho(f_1+f_2,g_1+g_2)ad(\phi(g_1),\phi(g_2))ad(f_1,f_2)g_3.
\end{align*}
The other conditions prove similarly.
\end{example}
\begin{definition}
Let $(A,[.,.,.],\rho,\phi)$ be a 3-Hom-$\rho$-Lie algebra. Consider the triple $(V,\beta,\cdot)$ consisting of a $G$-graded vector space $V$, an even homomorphism $\beta$ of vector spaces and a linear operation $\cdot:\mathcal{L}\times V\longrightarrow V$ such that $\mathcal{L}_{g^{\prime}}\cdot V_h\subseteq V_{g^{\prime}+h}$ for all $g^{\prime},h\in G$. Then $(V,\beta,\cdot)$ is called  an $A$-module if
\begin{align}
[(f_1,f_2),(g_1,g_2)]_{_\mathcal{L}}\beta(m)&=\phi_1(f_1,f_2)\cdot((g_1,g_2)\cdot m) -\rho(f_1+f_2,g_1+g_2)\phi_1(g_1,g_2)\cdot((f_1,f_2)\cdot m),\label{3}\\
([f_1,f_2,f_3],\phi(g))\cdot \beta(m)&=\phi_1(f_1,f_2)\cdot ((f_3,g)\cdot m)+\rho(f_1,f_2+f_3)\phi_1(f_2,f_3)\cdot ((f_1,g)\cdot m)\label{4}\\
&\ \ \ +\rho(f_1+f_2,f_3)\phi_1(f_3,f_1)\cdot((f_1,g)\cdot m),\nonumber\\
(\phi(f),[g_1,g_2,g_3])\cdot \beta(m)&=\rho(f,g_1+g_2)\phi_1(g_1,g_2)\cdot ((f_2,g_3)\cdot m)\label{44}\\
&\ \ \ +\rho(f,g_2+g_3)\rho(g_1,g_2+g_3)\phi_1(g_2,g_3)\cdot ((f,g_1)\cdot m)\nonumber\\
&\ \ \  +\rho(f,g_1+g_3)\rho(g_1+g_2,g_3)\phi_1(g_3,g_1)\cdot((f,g_2)\cdot m),\nonumber\\
\phi_1(f_1,f_2)((g_1,g_2)\cdot m)&=\rho(f_1+f_2,g_1)(\phi(g_1),[f_1,f_2,g_1])\cdot\beta(m)+([f_1,f_2,g_1],\phi(g_2))\cdot\beta(m)\\
&\ \ \ +\rho(f_1+f_2,g_1+g_2)\phi_1(g_1,g_2)\cdot((f_1,f_2)\cdot m).\nonumber
\end{align}
\end{definition} 
\begin{lemma}
Let $(A,[.,.,.],\rho,\phi)$ be a 3-Hom-$\rho$-Lie algebra. The modules and the representations of $A$ are equivalent.
\end{lemma}
\begin{proof}
Let $(V,\mu,\beta)$ be a representation of $A$. We define the linear operation $.:\mathcal{L}\times V\longrightarrow V$ by $(f,g,m)\mapsto(f,g)\cdot m=\mu(f,g)(m)$ and show that $(V,\beta,\cdot)$ is an $A$-module. For this purpose, it is sufficient to check two relations \eqref{3} and \eqref{4}. Then we have
\begin{align*}
[(f_1,f_2),(g_1,g_2)]_{\mathcal{L}}\cdot \beta(m)&=\mu([(f_1,f_2),(g_1,g_2)]_{\mathcal{L}})\beta(m)=\mu(\phi_1(f_1,f_2))\mu(g_1,g_2)m\\
&\ \ \ -\rho(f_1+f_2,g_1+g_2)\mu(\phi_1(g_1,g_2))\mu(f_1,f_2)m\\
&=\phi_1(f_1,f_2)\cdot((g_1,g_2)\cdot m) -\rho(f_1+f_2,g_1+g_2)\phi_1(g_1,g_2)\cdot((f_1,f_2)\cdot m).
\end{align*} 
For the next condition, we have
\begin{align*}
([f_1,f_2,f_3],\phi(g))\cdot \beta(m)&=\mu([f_1,f_2,f_3],\phi(g)) \beta(m)=\mu(\phi_1(f_1,f_2))\mu(f_3,g)m\\
&\ \ \ +\rho(f_1,f_2+f_3)\mu(\phi_1(f_2,f_3))\mu(f_1,g)m+\rho(f_1+f_2,f_3)\mu(\phi_1(f_3,f_1))\mu(f_2,g)m\\
&=\phi_1(f_1,f_2)\cdot ((f_3,g)\cdot m)+\rho(f_1,f_2+f_3)\phi_1(f_2,f_3)\cdot ((f_1,g)\cdot m)\\
&\ \ \ +\rho(f_1+f_2,f_3)\phi_1(f_3,f_1)\cdot((f_1,g)\cdot m).
\end{align*}
In the same method, we can check the last two conditions.\\
For the converse, consider the linear map $\mu:A\times A\longrightarrow {\rm End_G(V)}$ by $\mu(f,g)m=(f,g)\cdot m$. It is easy to see that for an even homomorphism $\beta:V\longrightarrow V$, the triple $(V,\beta,\cdot)$ is an $A$-module.
\end{proof}
\begin{definition}
Let $A$ be a 3-Hom-$\rho$-Lie algebra and $(V,\mu)$ be an $A$-module. An $n$-cochain on $A$ is a $\rho$-skew symmetric morphism $\omega$ from $\wedge^{2n+1}(A)$ into $V$ of degree $|\omega|$. Let us denote by $C^n(A,V)$ the set of all $n$-cochains on $A$, in the sense of 
$$C^n(A,V)=Hom(\wedge^{2n+1}(A), V), \quad \omega(f_1,\cdots, f_{2n+1})\in V_{|f_1|+\cdots |f_{2n+1}|+|\omega|},$$
where $f_1,\cdots,f_{2n+1} \in Hg(A)$.
$\omega\in C^n(A,V)$ is called an $n$-Hom-cochain on $A$ if for $f_1,\cdots,f_{2n+1} \in Hg(A)$ and $\beta\in {\rm End_G(V)}$, the following relation holds
$$\beta(\omega(f_1,\cdots,f_{2n+1})) = \omega(\phi(f_1),\cdots, \phi(f_{2n+1})).$$
We denote by $C^n_{\phi}(A,V)$ the set of all $n$-Hom-cochains on $A$.\\
Let $A$ be a multiplication 3-Hom-$\rho$-Lie algebra and $\beta=Id_V$. Define the coboundary operator $d_{n-1}:C^{n-1}_{\phi}(A,V)\longrightarrow C^n_{\phi}(A,V)$ by
\begin{align*}
d_{n-1}\omega(f_1,\cdots,f_{2n+1})&=(-1)^{n+1}\rho(f_1+\cdots+f_{2n-2}, f_{2n-1}+f_{2n+1})\rho(f_{2n-1}+f_{2n},f_{2n+1})\\
&\ \ \ \times\mu(\phi^{n-1}(f_{2n+1}), \phi^{n-1}(f_{2n-1}))\omega(f_1,\cdots,f_{2n-2},f_{2n})\\
&\ \ \ +(-1)^{n+1}\rho(f_1+\cdots+f_{2n-2}, f_{2n}+f_{2n+1})\rho(f_{2n-1},f_{2n}+f_{2n+1})\\
&\ \ \ \times\mu(\phi^{n-1}(f_{2n}), \phi^{n-1}(f_{2n+1}))\omega(f_1,\cdots,f_{2n-1})\\
&\ \ \ +\sum_{k=1}^n (-1)^{k+1}\rho(f_1+\cdots+f_{2k-2}, f_{2k-1}+f_{2k})\\
&\ \ \ \times\mu(\phi^{n-1}(f_{2k-1}), \phi^{n-1}(f_{2k}))\omega(f_1,\cdots,\widehat{f_{2k-1}},\widehat{f_{2k}},\cdots,f_{2n+1})\\
&\ \ \ +\sum_{k=1}^n\sum_{j=2k+1}^{2n+1}(-1)^k\rho(f_1+\cdots+f_{2k}, f_{2k+1}+\cdots+f_{j-1})\\
&\ \ \ \times\omega(\phi(f_1),\cdots,\widehat{\phi(f_{2k-1})},\widehat{\phi(f_{2k})},\cdots,[f_{2k-1},f_{2k},f_j],\cdots, \phi(f_{2n+1})),
\end{align*}
for $n\geq 1$, and $\omega\in C^n_{\phi}(A,V)$, where $\widehat{\phi(f_{2k-1})}$ means that $\phi(f_{2k-1})$ is omitted. Note that $| d\omega|=|\omega|$ and $d_{n}\circ d_{n-1}=0$ (the condition $d_{n}\circ d_{n-1}=0$ does note follow if the condition $\omega\circ\phi=\omega$ is omitted, so it is necessary to define the differential operator on $n$-Hom-cochains).\\
For $n=1$:
\begin{align*}
d_0\omega(f_1,f_2,f_3)&=\mu(f_1,f_2)\omega(f_3)+\rho(f_1+f_2,f_3)\mu(f_3,f_1)\omega(f_2)\\
&\ \ \ +\rho(f_1,f_2+f_3)\mu(f_2,f_3)\omega(f_1)-\omega([f_1,f_2,f_3]).
\end{align*}
For $n=2$:
\begin{align*}
d_1\omega(f_1,f_2,g_1,g_2,g_3)&=\omega(\phi(f_1),\phi(f_2), [g_1,g_2,g_3])+\mu(\phi(f_1),\phi(f_2))\omega(g_1,g_2,g_3)\\
&\ \ \ -\omega([f_1,f_2,g_1],\phi(g_2),\phi(g_3))+\rho(f_1+f_2,g_1)\omega(\phi(g_1),[f_1,f_2,g_2],\phi(g_3))\\
&\ \ \ -\rho(f_1+f_2,g_1+g_2)\omega(\phi(g_1),\phi(g_2),[f_1,f_2,g_3])\\
&\ \ \ -\rho(f_1+f_2,g_2+g_3)\rho(g_1,g_2+g_3)\mu(\phi(g_2),\phi(g_3))\omega(f_1,f_2,g_1)\\
&\ \ \ -\rho(f_1+f_2,g_1+g_3)\rho(g_1+g_2,g_3)\mu(\phi(g_3),\phi(g_1))\omega(f_1,f_2,g_2)\\
&\ \ \ -\rho(f_1+f_2,g_1+g_2)\mu(\phi(g_1),\phi(g_2))\omega(f_1,f_2,g_3).
\end{align*}
\end{definition}
\begin{definition}\label{1234}
Let $A$ be a 3-Hom-$\rho$-Lie algebra and $(V,\mu)$ be an $A$-module. Then a morphism $\nu\in {\rm Hom_G(A,V)}$ is called 0-Hom-cocycle if and only if $d_0\nu=0$, in the other word
\begin{align*}
\mu(f_1,f_2)\nu(f_3)&+\rho(f_1+f_2,f_3)\mu(f_3,f_1)\nu(f_2)\\
&+\rho(f_1,f_2+f_3)\mu(f_2,f_3)\nu(f_1)=\nu([f_1,f_2,f_3]).
\end{align*}
Also, $\omega\in {\rm Hom(\wedge^3A,V)}$ is called 1-Hom-cocycle with respect to $\mu$ if and only if $d_1\omega=0$.
\end{definition}
\begin{definition}
Let $(A,[.,.,.],\rho,\phi)$ be a multiplicative 3-Hom-$\rho$-Lie algebra. For any non-negative integer $k$, denote the $k$ times composition of $\phi$ by $\phi^k$ ($\phi^k=\phi\circ\cdots\circ\phi$ ($k$ times)) such that $\phi^0=id$ and $\phi^1=\phi$. A $\phi^k$-$\rho$-derivation of degree $|X|$ on $A$ is a linear map $X : A \longrightarrow A$ 
such that\\
$$X\circ\phi=\phi\circ X\quad i.e.,\quad [X,\phi]_{\rho}=0,$$
and
\begin{equation}\label{5}
X[f,g,h]= [X(f),\phi^k(g),\phi^k(h)]+ \rho(X,f)[\phi^k(f),X(g),\phi^k(h)]+\rho(X,f+g)[\phi^k(f),\phi^k(g),X(h)],
\end{equation}
for all $f,g,h\in A$. Let us denote by $\rho\text{-}Der_{\phi^k} A$ the set of all $\phi^k$-$\rho$-derivations of $A$. 
\end{definition}
\begin{example}
We define the even linear map $ad_k(f_1,f_2):A\longrightarrow A$ by $ad_k(f_1,f_2)(g)=[f_1,f_2,\phi^k(g)]$ for $g\in A$. If we assume that for any $f_1,f_2\in A$,
$\phi(f_1)=f_1$ and $\phi(f_2)=f_2$, then $ad_k(f_1,f_2)$ is a $\phi^{k+1}$-derivation.\\
If we check the accuracy of the equality \eqref{5}, the assertion follows. Thus, we have
\begin{align*}
ad_k(f_1,f_2)[f,g,h]&=[f_1,f_2,\phi^k[f,g,h]]=[\phi(f_1),\phi(f_2),\phi^k[f,g,h]]\\
&=[\phi(f_1),\phi(f_2),[\phi^k(f),\phi^k(g),\phi^k(h)]]\\
&=[[f_1,f_2,\phi^k(f)],\phi^{k+1}(g),\phi^{k+1}(h)]\\
&\ \ \ +\rho(f_1+f_2,f)[\phi^{k+1}(f),[f_1,f_2,\phi^k(g)],\phi^{k+1}(h)]\\
&\ \ \ +\rho(f_1+f_2,f+g)[\phi^{k+1}(f),\phi^{k+1}(g),[f_1,f_2,\phi^k(h)]]\\
&=[ad_k(f_1,f_2)(f),\phi^{k+1}(g),\phi^{k+1}(h)]\\
&\ \ \ +\rho(f_1+f_2,f)[\phi^{k+1}(f),ad_k(f_1,f_2(g)),\phi^{k+1}(h)]\\
&\ \ \ +\rho(f_1+f_2,f+g)[\phi^{k+1}(f),\phi^{k+1}(g),ad_k(f_1,f_2)(h)].
\end{align*}
$ad_k(f_1,f_2)$  is called an inner $\phi^{k+1}$-derivation. Denote by ${\rm Inn_{\phi^k}(A)}$ the set of inner $\phi^{k}$-derivation, i.e.,
$${\rm Inn_{\phi^k}(A)}=\{[f_1,f_2,\phi^{k-1}(\cdot)]|~~ f_1,f_2\in A,~~ \phi(f_i)=f_i~~ i=1,2\}.$$
\end{example}
\begin{definition}
Let $A$ be a 3-Hom-$\rho$-Lie algebra.\\
1: A linear map $X:A\longrightarrow A$ is said to be a homogeneous generalized $\phi^k$-derivation of degree $|X|$ of $A$, if there exist three linear maps $Y,Z,W:A\longrightarrow A$ such that 
$$[X,\phi]=0,~~[Y,\phi]=0,~~[Z,\phi]=0,~~[W,\phi]=0,$$
and
\begin{align*}
W[f,g,h]&=[X(f),\phi^k(g),\phi^k(h)]+\rho(X,f)[\phi^k(f), Y(g),\phi^k(h)]\\
&\ \ \ +\rho(X,f+g)[\phi^k(f),\phi^k(g),Z(h)],
\end{align*}
for all $f,g,h\in A$. We denote the set of all homogeneous generalized $\phi^k$-derivation of degree 
$| X|$ of $A$ by $GDer_{\phi^k}(A)$.\\
2: We call $X:A\longrightarrow A$ a homogeneous $\phi^k$-quasi derivation of degree $|X|$ of $A$, if there exists a linear map $Y:A\longrightarrow A$ such that
$$[X,\phi]=0,~~[Y,\phi]=0,$$
and
\begin{align*}
Y[f,g,h]&=[X(f),\phi^k(g),\phi^k(h)]+\rho(X,f)[\phi^k(f), X(g),\phi^k(h)]\\
&\ \ \ +\rho(X,f+g)[\phi^k(f),\phi^k(g),X(h)],
\end{align*}
for all $f,g,h\in A$. We denote the set of all homogeneous $\phi^k$-quasi derivation of degree 
$| X|$ of $A$ by $QDer_{\phi^k}(A)$.\\
3: We call $X:A\longrightarrow A$ a homogeneous $\phi^k$-centroid element of degree $|X|$ of $A$, if it satisfies for all $f,g,h\in Hg(A)$
\begin{align*}
X[f,g,h]&=[X(f),\phi^k(g),\phi^k(h)]=\rho(X,f)[\phi^k(f), X(g),\phi^k(h)]\\
&=\rho(X,f+g)[\phi^k(f),\phi^k(g),X(h)].
\end{align*}
We denote the set of all homogeneous $\phi^k$-centroid elements of degree 
$| X|$ of $A$ by $C_{\phi^k}(A)$.\\
4: $X:A\longrightarrow A$ is said to be a homogeneous $\phi^k$-quasi centroid element of degree $|X|$ of $A$, if it satisfies for all $f,g,h\in Hg(A)$
\begin{align*}
X[f,g,h]&=[X(f),\phi^k(g),\phi^k(h)]=\rho(X,f)[\phi^k(f), X(g),\phi^k(h)]\\
&=\rho(X,f+g)[\phi^k(f),\phi^k(g),X(h)].
\end{align*}
We denote the set of all homogeneous $\phi^k$-quasi centroid elements of degree 
$| X|$ of $A$ by $QC_{\phi^k}(A)$.
\end{definition}
\begin{proposition}
Let $(A,[.,.,.],\rho,\phi)$ be a multiplication 3-Hom-$\rho$-Lie algebra. If $X\in GDer_{\phi^k}(A)$ and $X^{\prime}\in C_{\phi^s}(A)$, then $X^{\prime}X\in GDer_{\phi^{k+s}}(A)$.
\end{proposition}
\begin{proof}
Since $X\in GDer_{\phi^k}(A)$, then there exist $Y,Z,W:A\longrightarrow A$ such that 
\begin{align*}
W[f,g,h]&=[X(f),\phi^k(g),\phi^k(h)]+\rho(X,f)[\phi^k(f), Y(g),\phi^k(h)]\\
&\ \ \ +\rho(X,f+g)[\phi^k(f),\phi^k(g),Z(h)].
\end{align*}
On the other hand, since $X^{\prime}\in C_{\phi^s}(A)$ we have
\begin{align*}
X^{\prime}W[f,g,h]&=X^{\prime}[X(f),\phi^k(g),\phi^k(h)]+\rho(X,f)X^{\prime}[\phi^k(f), Y(g),\phi^k(h)]\\
&\ \ \ +\rho(X,f+g)X^{\prime}[\phi^{k+s}(f),\phi^{k+s}(g),Z(h)]\\
&=[X^{\prime}X(f),\phi^{k+s}(g),\phi^{k+s}(h)]+\rho(X,f)[\phi^{k+s}(f), X^{\prime}Y(g),\phi^{k+s}(h)]\\
&\ \ \ +\rho(X,f+g)[\phi^{k+s}(f),\phi^{k+s}(g),X^{\prime}Z(h)].
\end{align*}
Therefore, $X^{\prime}X\in GDer_{\phi^{k+s}}(A)$.
\end{proof}
\begin{proposition}
Let $(A,[.,.,.],\rho,\phi)$ be a multiplication 3-Hom-$\rho$-Lie algebra and $X\in C_{\phi^{k}}(A)$. Then $X$ is a $\phi^k$-quasi derivation of $A$. 
\end{proposition}
\begin{proof}
Assuming that $f,g,h \in Hg(A)$. So, we have
\begin{align*}
[X(f),\phi^k(g),\phi^k(&h)]+\rho(X,f)[\phi^k(f), X(g),\phi^k(h)]\\
&\ \ \ +\rho(X,f+g)[\phi^k(f),\phi^k(g),X(h)]\\
&=3[X(f),\phi^k(g),\phi^k(h)]=3X[f,g,h]\\
&=X^{\prime}[f,g,h].
\end{align*}
\end{proof}
\subsection{The Coadjoint Representation}
We consider $A^{\star}$ as a dual space of $A$, $A^{\star}$ is a $G$-graded space, where $A^{\star}_a=\{\alpha\in A^{\star}| \alpha(f)=0,~~\forall~~f:~~|f|\neq -a\}$. Moreover, $A^{\star}$ is a graded $A$-module. Also, since $A=\oplus_{a\in G}A_a$ and $A^{\star}=\oplus_{a\in G}A^{\star}_a$ are $G$-graded spaces, the direct sum 
$$A\oplus A^{\star}=\oplus_{a\in G}(A\oplus A^{\star})_a=\oplus_{a\in G}(A_a\oplus A^{\star}_a),$$
is $G$-graded. Consider a homogeneous element of $A\oplus A^{\star}$ as $f+\alpha$ such that $f\in A$ and $\alpha\in A^{\star}$, with $|f+\alpha|=|f|=|\alpha|$.

Let $(A,[.,.,.]_A,\rho,\phi)$ be a 3-Hom-$\rho$-Lie algebra and $(V,\mu,\beta)$ be a representation of $A$. Let $V^{\star}$ be the dual vector space of $V$. We define the linear map $\widetilde{\mu}:A\times A\longrightarrow End(V^{\star})$ by $\widetilde{\mu}(f_1,f_2)(\varrho)=-\rho(f_1+f_2,\varrho)\varrho\circ\mu(f_1,f_2)$, where $f_1,f_2\in A,~~\varrho\in V^{\star}$ and set $\widetilde{\beta}(\varrho)=\varrho\circ\beta$.
\begin{proposition}
Let $(A,[.,.,.]_A,\rho,\phi)$ be a 3-Hom-$\rho$-Lie algebra and $(V,\mu,\beta)$ be a representation of $A$. Then the triple $(V^{\star},\widetilde{\mu},\widetilde{\beta})$ defines a representation of $A$ if and only if 
\begin{align*}
&(1).~~\beta(\mu[(f_1,f_2),(g_1,g_2)]_{\mathcal{L}})=\mu(f_1,f_2)\mu(\phi(g_1),\phi(g_2))-\rho(f_1+f_2,g_1+g_2)\mu(g_1,g_2)\mu(\phi(f_1),\phi(f_2)),\\
&(2).~~ \beta\mu([g_1,g_2,g_3],\phi(f))=-\rho(g_1+g_2,g_3+f)\mu(g_3,f)\mu(\phi(g_1),\phi(g_2))\\
&\qquad\qquad\qquad\qquad\qquad\qquad -\rho(g_2+g_3,g_1+f)\mu(g_1,f)\mu(\phi(g_2),\phi(g_3))\\
&\qquad\qquad\qquad\qquad\qquad\qquad-\rho(g_3+g_1,g_2+f)\mu(g_2,f)\mu(\phi(g_3),\phi(g_1)),\\
&(3).~~ \beta\mu(\phi(g),[f_1,f_2,f_3])=-\rho(f_1+f_2,f_3)\mu(g,f_3)\mu(\phi(f_1),\phi(f_2))-\mu(g,f_1)\mu(\phi(f_2),\phi(f_3))\\
&\qquad\qquad\qquad\qquad\qquad\qquad-\rho(f_1+f_2,f_3)\rho(f_1+f_3,f_2)\mu(g,f_2)\mu(\phi(f_3),\phi(f_1)),\\
&(4).~~  \rho(f_1+f_2,g_1)\beta\mu(\phi(g_1),[f_1,f_2,g_2])+\beta\mu([f_1,f_2,g_1],\phi(g_2))=\mu(f_1,f_2)\mu(\phi(g_1),\phi(g_2))\\
&\qquad\qquad\qquad\qquad\qquad\qquad-\rho(f_1+f_2,g_1+g_2)\mu(g_1,g_2)\mu(\phi(f_1),\phi(f_2)).
\end{align*} 
\end{proposition}
\begin{proof}
Assuming that the conditions (1)-(4) hold. 
We prove that $\widetilde{\mu}$ is a representation of $A$, thus we must check the properties \eqref{1}-\eqref{7.1} for $(\widetilde{\mu},\widetilde{\beta},V^{\star})$. For this, we check the property \eqref{1} and the others will prove similarly. Using (1), we start with the computation of the left-hand side of \eqref{1}:
\begin{align*}
\widetilde{\mu}[(f_1,f_2), (g_1,g_2)]_{\mathcal{L}}\widetilde{\beta}(\varrho)(m)&=-\rho(f_1+f_2+g_1+g_2,\varrho)(\varrho\circ\beta)(\mu[(f_1,f_2), (g_1,g_2)]_{\mathcal{L}})\\
&=\rho(f_1+f_2+g_1+g_2,\varrho)\rho(f_1+f_2,g_1+g_2)\varrho\mu(g_1,g_2)\mu(\phi(f_1),\phi(f_2))(m)\\
&\ \ \ -\rho(f_1+f_2+g_1+g_2,\varrho)\varrho\mu(f_1,f_2)\mu(\phi(g_1),\phi(g_2))(m)\\
&=\widetilde{\mu}(\phi(f_1),\phi(f_2))\widetilde{\mu}(g_1,g_2)\varrho(m)\\
&\ \ \ -\rho(f_1+f_2,g_1+g_2)\widetilde{\mu}(\phi(g_1),\phi(g_2))\widetilde{\mu}(f_1,f_2)\varrho(m).
\end{align*}
Therefore
\begin{align*}
\widetilde{\mu}[(f_1,f_2), (g_1,g_2)]_{\mathcal{L}}\circ\widetilde{\beta}&=\widetilde{\mu}(\phi(f_1),\phi(f_2))\widetilde{\mu}(g_1,g_2)\\
&\ \ \ -\rho(f_1+f_2,g_1+g_2)\widetilde{\mu}(\phi(g_1),\phi(g_2))\widetilde{\mu}(f_1,f_2).
\end{align*}
The proof of the converse is also straightforward.
\end{proof}
\begin{corollary}
Let $(A,[.,.,.]_A,\rho,\phi)$ be a 3-Hom-$\rho$-Lie algebra with the adjoint representation and $A^{\star}$ be the dual of $A$. The linear map $ad^{\star}:A\times A\longrightarrow {\rm End(A^{\star})}$ defined by $ad^{\star}(f_1,f_2)(\varrho)(f_3)=-\rho(f_1+f_2,\varrho)\varrho[f_1,f_2,f_3]_A=-\rho(f_1+f_2,\varrho)\varrho(ad(f_1,f_2)(f_3)$ for $f_1,f_2,f_3\in A$ and $\varrho\in A^{\star}$ is a representation of $A$ that is called coadjoint representation.
\end{corollary}
\begin{theorem}
Let $(A,[.,.,.]_A,\rho,\phi)$ be a 3-Hom-$\rho$-Lie algebra with the adjoint representation and $A^{\star}$ be the dual of $A$. Consider the linear map $\omega:A\times A\times A\longrightarrow A^{\star}$ and let $\widetilde{\mu}=ad^{\star}$. The $G$-graded space $A\oplus A^{\star}$ together with the bracket and linear map 
\begin{align*}
[f_1+\alpha_1,f_2+\alpha_2,f_3+\alpha_3]_{A\oplus A^{\star}}&=[f_1,f_2,f_3]_A+\omega(f_1,f_2,f_3)+\widetilde{\mu}(f_1,f_2)(\alpha_3)\\
&\ \  \ +\rho(f_1+f_2,f_3)\widetilde{\mu}(f_3,f_2)(\alpha_2)+\rho(f_1,f_2+f_3)\widetilde{\mu}(f_2,f_3)(\alpha_1),\\
\phi^{\prime}(f+\alpha)=\phi(f)+\alpha\circ\phi,\ \ \ \ \ \ \ &
\end{align*}
for all $f_i\in A$ and $\alpha_i\in A^{\star},~~~ i=1,2,3$, is a 3-Hom-$\rho$-Lie algebra if and only if $\omega$ is a 1-Hom-cocycle with respect to $\widetilde{\mu}$.
\end{theorem}
\begin{proof}
Using \eqref{2}-\eqref{7.1} and Definition \ref{1234}, the result easily follows.
\begin{definition}
The 3-hom-$\rho$-Lie algebra $(A\oplus A^{\star}, [.,.,.]_{A\oplus A^{\star}} , \phi^{\prime})$ is called the $T^{\star}_{\omega}$-extension of $(A, [.,.,.]_A, \phi)$.
\end{definition}
\end{proof}
\section{abelian extension of 3-Hom-$\rho$-lie algebra}
In this section, we discuss about extensions and abelian extensions of 3-Hom-$\rho$-Lie algebra $A$ and show that associated to any abelian extension, there is a
representation and a 2-cocycle. We assume that the 3-Hom-$\rho$-Lie algebra $A$ is multiplicative.
\begin{definition}
A sub-vector space $I\subseteq A$ is called a Hom subalgebra of $(A,[.,.,.]_A,\phi)$ if $[I,I,I]_A\subseteq I$ and $\phi(I)\subseteq I$, $I$ also is called a Hom ideal of $A$ if $\phi(I)\subseteq I$ and $[I,A,A]_A\subseteq I$. $I$ is said to be a Hom abelian ideal of $A$ if $[A,I,I]_A=0$.
\end{definition}
\begin{definition}
Let $(A,[.,.,]_A,\rho,\phi)$, $(V,[.,.,.]_V,\phi_V)$ and $(B,[.,.,.]_B,\psi)$ be three 3-Hom-$\rho$-Lie algebras and $i:V\longrightarrow B$, $p:B\longrightarrow A$ be homomorphisms. The sequence 
\begin{equation*}
\begin{tikzcd}
0\arrow{r} & V \arrow{r}{i} & B\arrow{r}{p} & A \arrow{r} & 0,
\end{tikzcd}
\end{equation*}
of 3-Hom-$\rho$-Lie algebras is a short exact sequence if ${\rm Im(i)}={\rm Ker(p)}$, ${\rm Ker(i)=0}$, ${\rm Im(p)}=A$ and $\phi_V(V)=\psi(V)$. In this case, $B$ is called an extension of $A$ by $V$ and denote it by $E_B$. Also, we call $B$ an abelian extension of $A$ if $V$ is a Hom abelian ideal of $B$ i.e., $[.,u,v]_B=0$
for all $u,v\in V$. A linear map $\delta:A\longrightarrow B$ is called a section of $p:B\longrightarrow A$ if $p\circ\delta=id_{A}$ and $\delta\circ\phi=\psi\circ\delta$.
\end{definition}
\begin{definition}
Two extensions $0\xrightarrow{} V\xrightarrow{i} B\xrightarrow{p} A\xrightarrow{} 0$ and 
$0\xrightarrow{} V\xrightarrow{j} B\xrightarrow{q} A\xrightarrow{} 0$ of 3-Hom-$\rho$-Lie algebra $A$ are equivalent if there exists a morphism $F:B\longrightarrow\tilde{B}$ of 3-Hom-$\rho$-Lie algebras such that the following diagram commutes: 
\begin{equation*}
\begin{tikzcd}
0\arrow{r} & V \arrow{r}{i}\arrow{d}{id}  & B\arrow{r}{p}\arrow{d}{F} & A \arrow{r}\arrow{d}{id} & 0\\
0\arrow{r} & V \arrow{r}{j} & \tilde{B}\arrow{r}{q}& A \arrow{r} & 0.
\end{tikzcd}
\end{equation*}
\end{definition}
\vspace{1 cm}
Let $B$ be an abelian extension of $A$ by $V$ and $\delta:A\longrightarrow B$ be a section. Define $\mu:\wedge^2A\longrightarrow{\rm End(V)}$ by 
\begin{equation}\label{t22}
\mu(f)(u)=\mu(f_1,f_2)(u)=[\delta(f_1),\delta(f_2),u]_B=ad(\delta(f))u,
\end{equation}
for all $f=(f_1,f_2)\in\wedge^2A, u\in V$.
\begin{proposition}
Let $(V,\phi_V)$, $(A,\phi)$ and $(B,\psi)$ be multiplication 3-Hom-$\rho$-Lie algebras and $B$ be an abelian extension of $A$ by $V$. Cosider $\mu$ with \eqref{t22}. Then, $(V,\phi_V,\mu)$ is a representation of $(A,\phi)$ and does not depend on the choice of the section $\delta$. Moreover, equivalent abelian extensions give the same representation.
\end{proposition}
\begin{proof}
At first, we show that $\mu$ is independent of the choice of $\delta$. If $\delta^{\prime}:A\longrightarrow B$ is another section, then 
$$p(\delta(f_i)-\delta^{\prime}(f_i))=f_i-f_i=0,$$
thus, $\delta(f_i)-\delta^{\prime}(f_i)\in V$. So $\delta^{\prime}(f_i)=\delta(f_i)+u$ for some $u\in V$. Since $[.,u,v]_B=0$ for all $u,v\in V$, we deduce that
\begin{align*}
[\delta^{\prime}(f_1),\delta^{\prime}(f_2),w]_B&=[\delta(f_1)+u,\delta(f_2)+v,w]_B\\
&=[\delta(f_1),\delta(f_2)+v,w]_B+[u,\delta(f_2)+v,w]_B\\
&=[\delta(f_1),\delta(f_2),w]_B+[\delta(f_1),v,w]_B+[u,\delta(f_2),w]_B+[u,v,w]_B\\
&=[\delta(f_1),\delta(f_2),w]_B.
\end{align*}
So, $\mu$ is independent of the choice of $\delta$. In the next, we show that $(V,\phi_V,\mu)$ is a representation of $(A,\phi)$. For this it is enough to check the conditions \eqref{1} and \eqref{2}. Let $\beta=\phi_V\in {\rm End(V)}$. By the third property of 3-Hom-$\rho$-Lie algebras, we have 
\begin{align*}
[\psi(\delta(f_1)),\psi(\delta(f_2)),[\delta(g_1),\delta(g_2),u]_B]_B&=[[\delta(f_1),\delta(f_2),\delta(g_1)]_B,\psi(\delta(g_2)),\psi(u)]_B\\
&\ \ \ +\rho(f_1+f_2,g_1)[\psi(\delta(g_1)),[\delta(f_1),\delta(f_2),\delta(g_2)]_B,\psi(u)]_B\\
&\ \ \ +\rho(f_1+f_2,g_1+g_2)[\psi(\delta(g_1)),\psi(\delta(g_2)),[\delta(f_1),\delta(f_2),u)]_B]_B.
\end{align*} 
Using \eqref{t22} and this fact that $\psi\circ\delta=\delta\circ\phi$, we have
\begin{align*}
\mu(\phi(f_1),\phi(f_2))\mu(g_1,g_2)u&=\mu[(f_1,f_2),(g_1,g_2)]_A\circ\phi_V(u)\\
&\ \ \ +\rho(f_1+f_2,g_1+g_2)\mu(\phi(g_1),\phi(g_2))\mu(f_1,f_2)u.
\end{align*}
Therefore
\begin{align*}
\mu[(f_1,f_2),(g_1,g_2)]_A\circ\phi_V(u)&=\mu(\phi(f_1),\phi(f_2))\mu(g_1,g_2)u\\
&\ \ \ -\rho(f_1+f_2,g_1+g_2)\mu(\phi(g_1),\phi(g_2))\mu(f_1,f_2)u.
\end{align*}
This gives us the condition \eqref{1}. In the continues, we try to prove the correctness of the condition \eqref{2}.\\
Since $\phi_V(V)=\psi(V)$, $\delta\circ\phi=\psi\circ\delta$, $[\delta(f_1),\delta(f_2),\delta(g_1)]_B-\delta[f_1,f_2,g_1]_A\in V$ and $V$ is abelian ideal then 
\begin{align}
\mu([f_1,f_2,g_1],\phi(g_2))\phi_V(u)&=[\delta[f_1,f_2,g_1],\delta(\phi(g_2)),\phi_V(u)]\nonumber\\
&=[[\delta(f_1),\delta(f_2),\delta(g_1)],\psi\circ\delta(g_2),\psi(u)].\label{231}
\end{align}
On the other hand, we have
\begin{align*}
[\psi(\delta(f_1)),\psi(u),[\delta(g_1),\delta(g_2),\delta(g_3)]_B]_B&=[[\delta(f_1),u,\delta(g_1)]_B,\psi(\delta(g_2)),\psi(\delta(g_3))]_B\\
&\ \ \ +\rho(f_1+u,g_1)[\psi(\delta(g_1)),[\delta(f_1),u,\delta(g_2)]_B,\psi(\delta(g_3))]_B\\
&\ \ \ +\rho(f_1+u,g_1+g_2)[\psi(\delta(g_1)),\psi(\delta(g_2)),[\delta(f_1),u,\delta(g_3)]_B]_B.
\end{align*}
By invoking \eqref{231} and this fact that $\delta\circ\phi=\psi\circ\delta$ and $\beta=\phi_V$, we conclude that
\begin{align*}
\rho(f_1+u,g_1+g_2+g_3)\mu(&[g_1,g_2,g_3],\phi(f_1))\phi_V(u)=\rho(f_1+u,g_1+g_2)\rho(f_1+u,g_3)\mu(\phi(g_1),\phi(g_2))\mu(g_3,f_1)u\\
& \ \ \ +\rho(f_1+u+g_1,g_2+g_3)\rho(f_1+u,g_1)\mu(\phi(g_2),\phi(g_3))\mu(g_1,f_1)u\\
&\ \ \ + \rho(f_1+u+g_2,g_3)\rho(g_1,g_3)\rho(f_1+u,g_1+g_2)\mu(\phi(g_3),\phi(g_1))\mu(g_2,f_1)u.
\end{align*}
So, this statements lead us to
\begin{align*}
\mu([g_1,g_2,g_3],\phi(f_1))\circ\phi_V&=\mu(\phi(g_1),\phi(g_2))\mu(g_3,f_1)\\
& \ \ \ +\rho(g_1,g_2+g_3)\mu(\phi(g_2),\phi(g_3))\mu(g_1,f_1)\\
&\ \ \ + \rho(g_1+g_2,g_3)\mu(\phi(g_3),\phi(g_1))\mu(g_2,f_1).
\end{align*}
Therefore, the result holds. At last, we investigate that equivalent abelian extension give the same representation. For this, suppose that $E_B$ and $E_{\tilde{B}}$ are equivalent abelian extensions presented by 
\begin{equation*}
\begin{tikzcd}
	0\arrow{r} & V \arrow{r}{i} & B\arrow{r}{p}& A \arrow{r} & 0\\
	0\arrow{r} & V \arrow{r}{j} & \tilde{B}\arrow{r}{q}& A \arrow{r} & 0,
\end{tikzcd}
\end{equation*}
 and $F:B\longrightarrow \tilde{B}$ is the 3-Hom-$\rho$-Lie algebra homomorphism, satisfying $F\circ i=j$, $q\circ F=p$. Choose linear sections $\delta$ and $\delta^{\prime}$ of $p$ and $q$. So, we obtain $q\circ F(\delta(f_i))=p\circ\delta(f_i)=f_i=q\circ\delta^{\prime}(f_i)$. Then, $F\circ\delta(f_i)-\delta^{\prime}(f_i)\in{\rm Ker(q)}\cong V$. Thus, we have
 $$[\delta(f_1), \delta(f_2),u]_B=[F\circ\delta(f_1),F\circ\delta(f_2),u]_{\tilde{B}}=[\delta^{\prime}(f_1),\delta^{\prime}(f_2),u]_{\tilde{B}}.
 $$
 Therefore, we get the result.
\end{proof}
\begin{proposition}
Let $\delta:A\longrightarrow B$ be a section of the abelian extension of $A$ by $V$. Define the map 
$$\omega(f_1,f_2,f_3)=[\delta(f_1),\delta(f_2),\delta(f_3)]_B-\delta([f_1,f_2,f_3]_A),$$
for all $f_1, f_2, f_3\in A$. Then $\omega$ is a 1-cocycle, where the representation $\mu$ is given by \eqref{t22}.
\end{proposition}
\begin{proof}
Since $B$ is a 3-Hom-$\rho$-Lie algebra, we have
\begin{align}\label{t25}
[\psi(\delta(f_1)), \psi(\delta(f_2)), [\delta(g_1), \delta(g_2), \delta(g_3)&]_B]_B=[[\delta(f_1), \delta(f_2), \delta(g_1)]_B, \psi(\delta(g_2)), \psi(\delta(g_3))]_B\nonumber\\
&\ \ \  +\rho(f_1+f_2,g_1)[\psi(\delta(g_1)), [\delta(f_1), \delta(f_2), \delta(g_2)]_B, \psi(\delta(g_3))]_B\nonumber\\
&\ \ \  +\rho(f_1+f_2+g_1+g_2)[\psi(\delta(g_3)), \psi(\delta(g_3)), [\delta(f_1), \delta(f_2), \delta(g_3)]_B ]_B.
\end{align}
On the other hand, we have
\begin{align*}
[\psi(\delta(f_1)), \psi(\delta(f_2)), [\delta(g_1), \delta(g_2), \delta(g_3)&]_B]_B=[\psi(\delta(f_1)), \psi(\delta(f_2)), \omega(g_1, g_2, g_3)+\delta[g_1, g_2, g_3]_A]_B\\
&= \mu(\phi(f_1), \phi(f_2))\omega(g_1, g_2, g_3) +[\psi(\delta(f_1)), \psi(\delta(f_2)),\delta[g_1, g_2, g_3]_A]_B\\
&= \mu(\phi(f_1), \phi(f_2))\omega(g_1, g_2, g_3) +\omega(\phi(f_1), \phi(f_2), [g_1, g_2, g_3]_A)\\
&\ \ \ +\delta[\psi(\delta(f_1)), \psi(\delta(f_2)),[g_1, g_2, g_3]_A]_B.
\end{align*}
Similarly, the right hand side of \eqref{t25} is equal to
\begin{align*}
&\rho(f_1+f_2+g_1, g_2+g_3)\mu(\phi(g_2), \phi(g_3))\omega(f_1, f_2, g_1)+\omega([f_1, f_2, g_1]_A,\phi(g_2), \phi(g_3))\\
&\ \ \ +\delta[[f_1, f_2, g_1]_A, \phi(g_2), \phi(g_3)]_A+\rho(f_1+f_2, g_1)\rho(f_1+f_2+g_2, g_3)\rho(g_1, g_2)\mu(\phi(g_3), \phi(g_1))\omega(f_1, f_2, g_2)\\
&\ \ \ +\rho(f_1+f_2, g_1)\delta[\phi(g_1), [f_1, f_2, g_2]_A, \phi(g_3)]_A+\rho(f_1+f_2, g_1+g_2)
\mu(\phi(g_1), \phi(g_2))\omega(f_1, f_2, g_3)\\
&\ \ \ +\rho(f_1+f_2, g_1+g_2)\omega(\phi(g_1), \phi(g_2), [f_1, f_2, g_3]_A) +\rho(f_1+f_2, g_1+g_2)\delta[\phi(g_1), \phi(g_2), [f_1, f_2, g_3]_A]_A.
\end{align*}
So, we get
\begin{align*}
\omega(\phi(f_1), \phi(f_2), [g_1, g_2, g_3]_A) &+\mu(\phi(f_1), \phi(f_2))\omega(g_1, g_2, g_3)=\omega([f_1, f_2, g_1]_A,\phi(g_2), \phi(g_3))\\
&\ \ \ +\rho(f_1+f_2, g_1)\omega(\phi(g_1),[f_1, f_2, g_2]_A,\phi(g_3))\\
&\ \ \ +\rho(f_1+f_2, g_1+g_2)\omega(\phi(g_1),\phi(g_3), [f_1, f_2, g_3]_A)\\
&\ \ \ +\rho(f_1+f_2, g_1)\rho(f_1+f_2+g_2, g_3)\rho(g_1, g_3)
\mu(\phi(g_3), \phi(g_1))\omega(f_1, f_2, g_2)\\
&\ \ \ +\rho(f_1+f_2, g_1+g_2)
\mu(\phi(g_1), \phi(g_2))\omega(f_1, f_2, g_3)\\
&\ \ \ +\rho(f_1+f_2+g_1, g_2+g_3)\mu(\phi(g_2), \phi(g_3))\omega(f_1, f_2, g_1).
\end{align*}
Therefore, $\omega$ is a 1-cocycle.
\end{proof}
\section{Infinitesimal deformations of 3-Hom-$\rho$-Lie algebras}
In this section, we introduce infinitesimal deformations of 3-Hom-$\rho$-Lie algebras and define Hom-Nijenhuis operator of it.

Let $A$ be a 3-Hom-$\rho$-Lie algebra and $\omega:\wedge^3A\longrightarrow A$ be a morphism. Consider a $t$-parametrized family of linear operations
$$[f,g,h]_t=[f,g,h]_A+t\omega(f,g,h).$$
If $A$ with all the brackets $[.,.,.]_t$ endow regular 3-Hom-$\rho$-Lie algebra structure $(A, [.,.,.]_t, \rho,\phi)$ which is denoted by $A_t$, we
say that $\omega$ generates a $t$-parameter infinitesimal deformation of the 3-Lie colour algebra $A$.
\begin{theorem}
$\omega$ generates a $t$-parameter infinitesimal deformation of the 3-Hom-$\rho$-Lie algebra $A$ is equivalent to 
\begin{enumerate}
\item 
$\omega$ itself defines a 3-Hom-$\rho$-Lie algebra structure on $A$.
\item
$\omega$ is a 1-cocycle of $A$ with coefficients in the adjoint representation.
\end{enumerate}
\end{theorem}
\begin{proof}
Since $[.,.,.]_t$ endow $(A, [.,.,.]_t, \rho,\phi)$ a regular 3-Hom-$\rho$-Lie algebra structure, then, 
\begin{align*} 
[\phi(f_1),\phi(f_2),[g_1,g_2,g_3]_t]_t&=[[f_1,f_2,g_1]_t,\phi(g_2),\phi(g_3)]_t-\rho(f_1+f_2,g_1)[\phi(g_1),[f_1,f_2,g_2]_t,\phi(g_3)]_t\\
&\ \ \ +\rho(f_1+f_2, g_1+g_2)[\phi(g_1),\phi(g_2),[f_1,f_2,g_3]_t]_t.
\end{align*}
The left-hand side is equal to
\begin{align*}
[\phi(f_1),\phi(f_2),[g_1,g_2,g_3]_A]_A &+t\{\omega(\phi(f_1), \phi(f_2), [g_1, g_2, g_3]_A)+[\phi(f_1),\phi(f_2), \omega(g_1, g_2, g_3)]_A\}\\
& +t^2\omega(\phi(f_1), \phi(f_2), \omega(g_1, g_2, g_3)).
\end{align*}
The right-hand side too is equal to
\begin{align*}
&[[f_1, f_2, g_1]_A, \phi(g_1), \phi(g_3)]+t\omega([f_1, f_2, g_1]_A, \phi(g_2), \phi(g_3))+[t\omega(f_1,f_2,g_1),\phi(g_2),\phi(g_3)]_A\\
&+t^2\omega(\omega(f_1, f_2,g_1),\phi(g_2),\phi(g_3))+\rho(f_1+f_2, g_1)[\phi(g_1), [f_1, f_2,g_2]_A,\phi(g_3)]_A\\
&+\rho(f_1+f_2, g_1)t\omega(\phi(g_1),[f_1,f_2,g_2]_A,\phi(g_3))\\
&+\rho(f_1+f_2, g_1)[\phi(g_1), t\omega(f_1,f_2,g_2),\phi(g_3)]_A+\rho(f_1+f_2, g_1)t^2\omega(\phi(g_1),\omega(f_1,f_2,g_2),\phi(g_3))\\
&+\rho(f_1+f_2,g_1+g_2)[\phi(g_1),\phi(g_2),[f_1,f_2,g_3]_A]_A+\rho(f_1+f_2, g_1+g_2)t\omega(\phi(g_1),\phi(g_2),[f_1,f_2,g_3]_A)\\
&+\rho(f_1+f_2, g_1+g_2)[\phi(g_1),\phi(g_2),t\omega(f_1,f_2,f_3)]_A+\rho(f_1+f_2, g_1+g_2)t^2\omega(\phi(g_1),\phi(g_2),\omega(f_1,f_2,g_3)).
\end{align*}
Thus, we have
\begin{align*}
\omega(\phi(f_1), \phi(f_2), [g_1, g_2, g_3]_A) +[f_1,f_2,\omega(g_1,g_2,g_3)]_A&=\omega([f_1,f_2,g_1]_A, \phi(g_2), \phi(g_3))\\
&+\rho(f_1+f_2, g_1)\omega(\phi(g_1), [f_1,f_2,g_2]_A, \phi(g_3))\\
&+\rho(f_1+f_2, g_1+g_2)\omega(\phi(g_1),\phi(g_2), [f_1,f_2,g_3]_A)\\
&+[\omega(f_1,f_2,g_1),\phi(g_2),\phi(g_3)]_A\\
&+\rho(f_1+f_2,g_1)[\phi(g_1),\omega(f_1,f_2,g_2),\phi(g_3)]_A\\
&+\rho(f_1+f_2,g_1+g_2)[\phi(g_1), \phi(g_2), \omega(f_1, f_2, g_3)]_A,
\end{align*}
and 
\begin{align*}
\omega(\phi(f_1),\phi(f_2),\omega(g_1,g_2,g_3))&=\omega(\omega(f_1,f_2,g_1),\phi(g_2),\phi(g_3))+\rho(f_1+f_2,g_1)\omega(\phi(g_1),\omega(f_1,f_2,g_2),\phi(g_3))\\ &\ \ \ +\rho(f_1+f_2,g_1+g_2)\omega(\phi(g_1),\phi(g_3),\omega(f_1,f_2,g_3)).
\end{align*}
Therefore, $\omega$ defines a 3-Hom-$\rho$-Lie algebra structure on $A$ and $\omega$ is a 1-cocycle of $A$ with the coefficient in the adjoint representation.
\end{proof}
An infinitesimal deformation is said to be trivial if there exists a grade-preserving map $N:A\longrightarrow A$
such that for $T_t = id + tN: A_t \longrightarrow A$ the following relation holds
$$T_t[f_1, f_2, f_3]_t = [T_tf_1, T_tf_2, T_tf_3].$$
\begin{definition}
A linear operator $N:A\longrightarrow A$ is called a Hom Nijenhuis operator if 
\begin{align}\label{t123}
N^2[f_1, f_2, f_3] &= N[Nf_1, f_2, f_3] + N[f_1, Nf_2, f_3] + N[f_1, f_2, Nf_3]\nonumber\\
&\ \ \ -([Nf_1, Nf_2, f_3] + [Nf_1, f_2, Nf_3] + [f_1, Nf_2, Nf_3]).
\end{align}
If we define $[.,.,.]_N$ by 
$$[f_1,f_2,f_3]_N= [Nf_1, f_2, f_3] + [f_1, Nf_2, f_3] + [f_1, f_2, Nf_3] -N[f_1, f_2, f_3],$$
then \eqref{t123} is equivalent to
$$N[f_1, f_2, f_3]_N = [Nf_1, Nf_2, f_3] + [Nf_1, f_2, Nf_3] + [f_1, Nf_2, Nf_3].$$
\end{definition}
\begin{theorem}
Let $N$ be a Nijenhuis operator for $A$. Setting 
$$\omega(f_1, f_2, f_3) = [f_1,f_2,f_3]_N,$$
then $\omega$ is an infinitesimal
deformation of $A$. Furthermore, this deformation is a trivial one.
\end{theorem}
\begin{proof}
By a direct calculation, we can see $d_1\omega=0$, therefore $\omega$ is a 1-cocycle of
$A$ with the coefficients in the adjoint representation. We must check the $\rho$-fundamental identity for $\omega$, that is 
\begin{align*}
\omega(\phi(f_1), \phi(f_2), \omega(g_1, g_2, g_3)) &=\omega(\omega(f_1, f_2, g_1), \phi(g_2), \phi(g_3))
+\rho(f_1 + f_2, g_1)\omega(\phi(g_1), \omega(f_1, f_2, g_2), \phi(g_3))\\
&\ \ \ +\rho(f_1 + f_2, g_1 + g_2)\omega(\phi(g_1), \phi(g_2), \omega(f_1, f_2, g_3)).
\end{align*}
This identity follows easily by a direct calculation, using the $\rho$-fundamental identity for $A$ and this fact that $N$ is a Hom Nijenhuis operator for $A$. Suppose that $T_t=id + tN$, then
\begin{align*}
T_t[f_1, f_2, f_3]_t &=id + tN([f_1, f_2, f_3]_t)=id + tN([f_1,f_2,f_3]_A+t\omega(f_1,f_2,f_3))\\ &=[f_1, f_2, f_3] + t\omega(f_1, f_2, f_3) + tN([f_1, f_2, f_3] + t\omega(f_1, f_2, f_3))\\
&= [f_1, f_2, f_3] + t(\omega(f_1, f_2, f_3) + N[f_1, f_2, f_3]) + t2N\omega(f_1, f_2, f_3),
\end{align*}
and
\begin{align*}
[T_tf_1, T_tf_2, T_tf_3] &= [f_1 + tNf_1, f_2 + tNf_2, f_3 + tNf_3]\\
&= [f_1, f_2, f_3] + t([Nf_1, f_2, f_3] + [f_1, Nf_2, f_3] + [f_1, f_2, Nf_3])\\
&\ \ \ +t2([Nf_1, Nf_2, f_3] + [Nf_1, f_2, Nf_3] + [f_1, Nf_2, Nf_3])
+t3[Nf_1, Nf_2, Nf_3].
\end{align*}
Then, we have
$$T_t[f_1, f_2, f_3]_t =[T_tf_1, T_tf_2, T_tf_3],$$
which implies that the infinitesimal deformation is trivial.
\end{proof}


\end{document}